\documentclass[a4paper, 12pt]{article}
\usepackage[british]{babel}
\usepackage{amsmath,amssymb,amsthm}
\usepackage{hyperref} 

\newtheorem{lemma}{Lemma}
\newtheorem{corollary}{Corollary}
\newtheorem{prop}{Proposition}
\newtheorem{theorem}{Theorem}
\newtheorem{remark}{Remark}

\def\dt{{\sf dt}}

\newcommand{\F}{{\mathcal{F}}}
\renewcommand{\P}{\mathbb{P}}
\newcommand{\Z}{\mathbb{Z}}
\newcommand{\R}{\mathbb{R}}
\newcommand{\eps}{\varepsilon}
\def\a{\alpha}
\def\b{\beta}
\def\g{\gamma}
\def\l{\lambda}
\newcommand{\E}{\mathbb{E}}
\newcommand{\Gen}{\mathsf{G}}

\begin{document}

\title{Boundary effect in competition processes}

\author{Vadim Shcherbakov\footnote{Department of Mathematics, Royal Holloway, University of London, UK. Email address: vadim.shcherbakov@rhul.ac.uk}
\and Stanislav Volkov\footnote{Centre for Mathematical Sciences, Lund University, Sweden. Email address: s.volkov@maths.lth.se}}
\maketitle
\begin{abstract}
This paper is devoted to studying the long-term behaviour of a continuous time Markov chain that can be interpreted as a pair of linear birth processes which evolve with a competitive interaction; as a special case, they include the famous Lotka-Volterra interaction. Another example of our process is related to urn models with ball removals. We show that, with probability one, the process eventually escapes to infinity by sticking to the boundary in a rather unusual way. 
\end{abstract}



\section{The model and results}
In this paper we study the long term behaviour of a continuous time Markov chain (CTMC) with values in $\Z_{+}^2$, where $\Z_{+}$ is the set of all non-negative integers, defined on a certain probability space with probability measure $\P$. The Markov chain jumps only to the nearest neighbours, and we consider two types of transition rates described below.
\begin{enumerate}
\item[] {\it Transition rates of type I.} Given the state $(x_1, x_2)\in \Z_{+}^2$ the Markov chain jumps to 
\begin{align}
\label{eqjumps1}
\begin{split}
(x_1+1,x_2) &\quad\text{with rate}\quad\l_1+\a_1 x_1,\\
(x_1,x_2+1) &\quad\text{with rate}\quad\l_2+\a_2 x_2,\\
(x_1-1,x_2) &\quad\text{with rate}\quad x_1g_1(x_2)\quad\text{if}\quad x_1>0,\\
(x_1,x_2-1) &\quad\text{with rate}\quad x_2g_2(x_1)\quad\text{if}\quad x_2>0,
\end{split}
\end{align}
where $\a_i, \l_i>0, i=1, 2$, and $g_i, i=1, 2$ are some non-negative functions. We call the Markov chain with transition rates of type I {\it a competition process with non-linear interaction (specified by functions~$g_1$ and~$g_2$)}. 
\item[] {\it Transition rates of type II.} Given the state $(x_1, x_2)\in \Z_{+}^2$ the Markov chain jumps to 
\begin{align}
\label{eqjumps2}
\begin{split}
(x_1+1,x_2) &\quad\text{with rate}\quad\l_1+\a_1 x_1,\\
(x_1,x_2+1) &\quad\text{with rate}\quad\l_2+\a_2 x_2,\\
(x_1-1,x_2) &\quad\text{with rate}\quad \b_1 x_2\quad\text{if}\quad x_1>0,\\
(x_1,x_2-1) &\quad\text{with rate}\quad \b_2 x_1\quad\text{if}\quad x_2>0,
\end{split}
\end{align}
where $\a_i, \l_i\geq 0, i=1,2$ and $\b_i>0, i=1,2$. We call the Markov chain with transition rates of type II {\it a competition process with linear interaction}.
\end{enumerate}
Competition processes with both non-linear and linear interaction belong to a class of competition processes introduced in~\cite{Reuter} as a natural two-dimensional analogue of the birth-and-death process in $\Z_{+}$. In~\cite{Reuter}, the competition process was defined as a CTMC with values in $\Z_{+}^2$, where the transitions are allowed only to the nearest neighbour states (see Section~\ref{appendix} below). This definition was generalised to a multidimensional case in~\cite{Iglehart1},~\cite{Iglehart2}. Some basic models of competition processes are discussed in~\cite{Anderson}. The term ``a competition process'' was apparently coined due to the fact that original examples of such processes were motivated by modelling a competition between populations (e.g. see~\cite[Examples~1 and~2]{Reuter} and references therein). One of the most known competition processes is the one specified by the famous Lotka-Volterra interaction. In our notation, the Lotka-Volterra interaction corresponds to functions $g_i(z)=z$, $i=1,2$ in the case of transition rates of type I. If, in addition, $\lambda_i=0$, $i=1,2$ in the Lotka-Volterra case, then we get a competition process given in~\cite[Example 1]{Reuter}.

The competition processes with both non-linear and linear interactions can be interpreted in terms of interacting birth-and-death processes. Indeed, if in both cases the death rates are equated to zero, that is $g_1=g_2\equiv 0$ in~\eqref{eqjumps1} and $\beta_1=\beta_2=0$ in~\eqref{eqjumps2}, then the corresponding Markov chain is formed by two independent linear birth processes with immigration. Non-zero death rates determine competitive interaction between the components. Therefore, the competition processes of interest can be naturally embedded into a more general technical framework of multivariate Markov processes formulated in terms of locally interacting birth-and-death processes. In the absence of interaction components of such a Markov process evolve as a collection of independent birth-and-death processes, which long term behaviour is well known. Namely, given a set of transition rates one can, in principle, determine whether the corresponding birth-and-death process is recurrent/positive recurrent, or transient/explosive, and compute various characteristics of the process. However, the presence of interaction can significantly change the collective behaviour of the system (e.g. see~\cite{SVJ},~\cite{MS} and~\cite{SV}).

Furthermore, note that a discrete time Markov chain, DTMC for short, corresponding to the competition process with linear interaction, can be regarded as an urn model with ball removals. In the symmetric case, that is $\a_1=\a_2$, $\b_1=\b_2$ and $\l_1=\l_2$, this DTMC is similar, in a sense, to Friedman's urn model. This similarity enabled us to adapt the Freedman's method for Friedman's urn model (\cite{Freedman}) in order to obtain a key fact in the proofs. We discuss this method in detail in Section~\ref{Freed} below. If both $\a_1=\a_2=0$ and $\l_1=\l_2=0$, then DTMC corresponding to the competition process with linear interaction coincides with the well-known OK Corral model (see e.g.~\cite{Kingman}). If $\a_1=\a_2=0$ and $\l_1,\l_2>0$ then the corresponding competition process can be interpreted as the OK Corral model with ``resurrection". 

Theorems~\ref{T1} and~\ref{T2} below are the main results of the paper. These theorems show that competition processes with both non-linear and linear interaction have similar, rather unusual, long term behaviour. Note that in the theorems  and later in the proofs 
we use the following notation:
$$
i^*=\begin{cases}
2,&\text{ if } i=1;\\
1,&\text{ if } i=2\\
\end{cases}
$$
(i.e. ``the other coordinate'').
\begin{theorem}
\label{T1}
Let $\xi(t)$ be a competition process with non-linear interaction (transition rates of type I) specified by functions $g_1, g_2:[0,\infty)\to [0, \infty)$. Assume that
\begin{itemize}
\item 
$g_1$ and $g_2$ are regularly varying functions with indexes $\rho_1>0$ and $\rho_2>0$ respectively, such that $g_1(0)=g_2(0)=0$ and $g_1(x), g_2(x)>0$ for all $x>0$;
\item 
$\a_i, \l_i>0$, $i=1,2$.
\end{itemize}
Then $\xi(t)$ is a non-explosive transient CTMC and $\P\left(B_1 \cup B_2\right)=1$, where 
$$
B_i=\left\{\xi_i(t)\to \infty,\ 0=\liminf_{t\to\infty} \xi_{i^*}(t)<\limsup_{t\to\infty} \xi_{i^*}(t)=1\right\},\, i=1,2.
$$
\end{theorem}

\begin{theorem}
\label{T2}
Let $\xi(t)$ be a competition process with linear interaction (transition rates of type II) specified by parameters $\a_i\geq 0$, $\l_i>0$ and $\b_i> 0$, $i=1,2$. Then $\xi(t)$ is a non-explosive transient CTMC and $\P\left(A_1 \bigcup A_2\right)=1$, where for $i=1,2$
\begin{align}
\nonumber
 A_i&=\left\{
\lim_{t\to \infty}\xi_i(t)=\infty, 
\ \liminf_{t\to\infty} \xi_{i^*}(t)=0, \
\limsup_{t\to\infty} \xi_{i^*}(t)=\kappa_{i^{\star}}\right\}\text{ and }\\
\label{Li}
\kappa_i&=\kappa_i(\a_{i^*})=\begin{cases}
1,&\text{ if } \a_{i^*}>0;\\
2,&\text{ if } \a_{i^*}=0.
\end{cases} 
\end{align}
\end{theorem}

The proof of each theorem consists of two parts. First, we show that, with probability bounded below {\it uniformly} over the initial position on the coordinate axes, the process sticks to the boundary of the quarter plane in the following way. Namely, one of the components of the process tends to infinity while the other component takes only values $0$ and $1$ ($0$, $1$ and $2$ in the special case of Theorem~\ref{T2}) oscillating infinitely often between these values, as described. This is what we call the boundary effect. This step of the proof is the subject of Lemma~\ref{lem_stay@axesT1} (Theorem~\ref{T1}) and Lemma~\ref{lem_stay@axesT2} (Theorem~\ref{T2}). In order to prove each lemma we construct a so-called Lyapunov function for the well-known transience criterion of a countable Markov chain (e.g. see~\cite[Theorem 2.5.8]{MPW}). This allows us to show that the Markov chain is confined to a strip along the boundary, as described. Then this fact is complemented by an argument based on the Borel-Cantelli lemma that gives the oscillation effect, i.e. the process transits from one level of the strip to another infinitely often. In both cases, this implies that the Markov chain under consideration is transient and, with positive probability, escapes to infinity in a certain way.

Intuitively, it seems rather clear that if the process is already at the boundary, then it prefers to stay near the boundary in the future. We use the Lyapunov function method to transform this intuition into a rigorous argument in Lemmas~\ref{lem_stay@axesT1} and~\ref{lem_stay@axesT2}. Although a direct probabilistic proof might be possible in proving both lemmas, we prefer to use the general method, which can be used in other cases, where a direct probabilistic argument might become cumbersome. For example, this is the case in the model in~\cite{MS}, where a similar boundary effect was originally observed for a pair of interacting birth-and-death processes. In fact, we borrow the idea of the construction of the Lyapunov function from that article. 

Another key step of the proof of both theorems consists of showing that the process hits the boundary with probability one. Note that in applications of competition processes to population modelling the hitting time is interpreted as the extinction time of one of the competing populations. Therefore, determining whether the hitting time is finite is of interest in its own right. Sufficient conditions for finiteness of the hitting time and its mean are given in~\cite{Reuter} for competition process in $\Z_{+}^2$. These conditions are rather restrictive, which is not surprising as these conditions were obtained for very general assumptions on the transition characteristic of the competition process. For example, it is possible to use them only in some special cases of competition processes in Theorems~\ref{T1} and~\ref{T2} (see a discussion in Section~\ref{appendix}). We use a direct probabilistic argument in order to show finiteness of the mean hitting time in the case of the competition process with non-linear interaction in Theorem~\ref{T1} and also in the case of the competition process with linear interaction in Theorem~\ref{T2} under assumption $\a_1\a_2<\b_1\b_2$. However, neither this argument nor results of~\cite{Reuter} can be applied in the case of the competition process with linear interaction (Theorem~\ref{T2}) under assumption $\a_1\a_2>\b_1\b_2$. In this particular case showing that the process hits the boundary almost surely is somewhat reminiscent to showing non-convergence to an unstable equilibrium in processes with reinforcement (e.g.\ urn models). Often the method of stochastic approximation is used to show such non-convergence (see e.g.~\cite{Pemantle} and references therein). Further, showing that the hitting time is finite in this case of the linear model is similar also to showing that a non-homogeneous random walk exits a cone, where the Lyapunov function method proved to be useful (e.g. see~\cite{MMW} and references therein). 

Although our model is similar to both urn models with ball removals, and to non-homogeneous random walks, we were unable to apply the above research techniques and used a different method instead. Our method is a modification of a method used in~\cite{Freedman} for studying Friedman's urn model. The original method consists of estimating moments of certain martingales related to the process of interest. The similarity of the competition process with linear interaction with Friedman's urn model allows us to adapt this idea (see Section~\ref{Freed} for details).

\section{Proofs}
In what follows, $\E$ denotes the expectation with respect to the probability measure $\P$.

\subsection{Proof of Theorem~\ref{T1}}\label{ProofT1}
\begin{lemma}
\label{lem_stay@axesT1}
There exists $\eps>0$, depending on the model parameters only, such that
$$
\inf_{x_1\geq 0}\P\left(\tilde A_1|\xi(0)=(x_1, 0)\right)\ge \eps
\,\text{ and }\, \inf_{x_2\geq 0}\P\left(\tilde A_2|\xi(0)=(0, x_2)\right)\ge \eps,
$$
where 
$
\tilde A_i=\{\xi_i(t)\to\infty \text{ and } \xi_{i^{\star}}(t)\in\{0,1\}, \forall t\ge 0\}.
$
\end{lemma}
\begin{proof}[Proof of Lemma~\ref{lem_stay@axesT1}.]
We prove only the first bound of the lemma, that is when the process starts at $\xi(0)=(x_1, 0)$. The other bound will follow by symmetry. In order to simplify notation we denote~$x=x_1$ and~$y=x_2$ in the rest of the proof. Given positive numbers $\nu$ and $\mu$ define the following function on $\Z_{+}^2\setminus{(0, 0)}$
\begin{equation}
\label{f1/3}
f(x,y)=\begin{cases}
x^{-\nu}-x^{-\mu}, &\quad\text{if } y=0;\\
x^{-\nu}, &\quad\text{if } y=1;\\
1, &\quad\text{if } y\ge 2.
\end{cases}
\end{equation}
In the rest of the proof of this lemma we assume that 
\begin{equation}
\label{numu}
0<\nu<\mu<\min(\rho_1, \rho_2).
\end{equation}
Denote $\Gen$ the generator of CTMC $\xi(t)$ with transition rates~\eqref{eqjumps1}. From state $(x,0)$, where $x>0$, transitions are possible only to states $(x+1, 0)$ and $(x, 1)$ with rates $\l_1+\a_1x$ and~$\l_2$ respectively. Therefore,
\begin{equation}
\label{x0}
\Gen f(x,0)=(\l_1+\a_1x)\left((x+1)^{-\nu}-x^{-\nu}-(x+1)^{-\mu}+x^{-\mu}\right)
+\l_2x^{-\mu}.
\end{equation}
Given $\gamma>0$, Taylor's expansion formula shows that 
\begin{equation}
\label{Taylor}
(x\pm1)^{-\gamma}-x^{-\gamma}=\mp\gamma x^{-1-\gamma}+o\left(x^{-1-\gamma}\right)
\end{equation}
for sufficiently large $x>0$. Applying this expansion for the polynomial terms on the right hand side of~\eqref{x0} we obtain that 
\begin{equation}
\label{super0}
\Gen f(x,0) \leq x^{-\nu}\left(-\a_1\nu+ (\a_1\mu +\l_2)x^{-\mu+\nu}+o(1)
\right)\leq 0,
\end{equation}
for all sufficiently large $x$, as $0<\nu<\mu$.

Next, given state $(x,1)$, where $x>0$, the Markov chain can jump only to states $(x+1, 1)$, $(x-1, 1)$, $(x, 2)$ and $(x, 0)$, 
and these jumps occur with rates $\l_1+\a_1x$, $x\cdot g_1(1)$, $\l_2+\a_2$ and $g_2(x)\cdot 1$ respectively. Therefore, 
\begin{equation*}
\begin{split}
\Gen f(x,1)&=(\l_1+\a_1x)\left((x+1)^{-\nu}-x^{-\nu}\right)+xg_1(1)\left((x-1)^{-\nu}-x^{-\nu}\right)\\
&+
(\l_2+\a_2)\left(1-x^{-\nu}\right) +g_2(x)\left(x^{-\nu}-x^{-\mu}-x^{-\nu}\right)
%
\\ &
=-g_2(x)x^{-\mu}+O(1)
\end{split}
\end{equation*}
by applying the expansion~\eqref{Taylor}. Recall that $g_2$ is a regularly varying function with index $\rho_2>0$, that is 
$g_2(x)=x^{\rho_2}l(x)$, where $l$ is a slowly varying function (e.g. see \cite{BGT} for definitions). Since $\mu<\min(\rho_1, \rho_2)$ (see~\eqref{numu}), we get that $g_2(x)x^{-\mu}=x^{\rho_2-\mu}l(x)\to \infty$ as $x\to \infty$. This results in
\begin{equation}
\label{super1}
 \Gen f(x,1)\leq 0
\end{equation}
for all sufficiently large $x$. Define the following stopping time 
$$
\sigma=\inf(t: \xi(t)\notin\{x\ge N+1 \text{ and } y\le 1\}),
$$
where integer $N$ is such that the bounds~\eqref{super0} and~\eqref{super1} hold for all $x>N$. These bounds imply that the random process $Z(t):=f(\xi(t\wedge \sigma))$ is a supermartingale. Since~$Z(t)\ge0$, we conclude that $Z(t)$ converges almost surely to a finite limit $Z_{\infty}$. Next, note that on the event $\{\sigma=\infty\}$ we must have $\xi_1(t)\to\infty$, otherwise, if $\limsup_{t\to\infty}\xi_1(t)=A<\infty$, then~$Z(t)$ cannot converge due to the fact that $f$ is not constant on set $\{0,1,\dots,A\}\times\{0,1\}$ which is irreducible for the chain. Consequently,
$$
Z_\infty=\begin{cases}
1\text{ or }f(N, 0)=N^{-\nu}-N^{-\mu}\text{ or } f(N, 1)=N^{-\nu},& \text{ if } \sigma<\infty,\\
0,& \text{ if } \sigma=\infty.
\end{cases}
$$
Assume that the initial position of the process is $(x, 0)$, where $x\ge N+1$. By the optional stopping theorem
$$
(N^{-\nu}-N^{-\mu})\P(\sigma<\infty)\le \E (Z_\infty)\le Z(0)=f(x, 0)=x_0^{-\nu}-x^{-\mu},
$$
so that 
$$
\P(\sigma<\infty)\le
\frac{f(x_0,0)}{f(N,0)}
\le \frac{f(N+1,0)}{f(N,0)}=1-\eps'<1
$$
for some $\eps'>0$, due to the monotonicity of the function $x^{-\nu}-x^{-\mu}$ for positive $x$. Thus, if $\xi(0)=(x,0)$, where $x\geq N+1$, then, with probability at least $\eps'$, the process $\xi(t)$ stays in set $\{N+1,N+2,\dots\}\times\{0,1\}$ forever. Further, for each initial position $(x, 0)$, where $x\in\{0,1,\dots,N\}$, with a strictly positive probability, the process reaches state $(N+1, 0)$ without exiting set $\{y=0,1\}\in\Z_{+}^2$ (e.g. by just jumping only to the right). Consequently, $\P(\sigma=\infty|\xi(t)=(x, 0))$ is bounded away from zero uniformly over $x\geq 0$. On this event $\xi_1(t)\to\infty$ a.s.
\end{proof}

\begin{lemma}[Lemma~7.3.6 in~\cite{MPW}]\label{lem7.3.6}
Let $Y_t\ge 0$, $t\ge 0$, be a process adapted to a filtration ${\cal G}_t$, $t\ge 0$, and let $T$ be a stopping time. Suppose that there exists $\varepsilon>0$ such that 
$$
\E\left[{\sf d}Y_t|{\cal G}_{t-}\right]\le -\varepsilon \dt \text{ on }\{t\le T\}.
$$
Then $\E [T|{\cal G}_0]\le Y_0/\varepsilon$.
\end{lemma}

\begin{lemma}
\label{L2}
Define $\tau=\inf\{t:\xi_1(t)=0\text{ or }\xi_2(t)=0\}$. Then $\tau$ is a.s.\ finite.
\end{lemma} 
\begin{proof}[Proof of Lemma~\ref{L2}.]
It is easy to see that the infinitesimal mean jump of component $\xi_i(t)$ computed as
$$
\E(\xi_i(t+\dt)-\xi_i(t)| \xi(t)=(x_1, x_2))=(\l_i+(\a_i-g_i(x_{i^{\star}}))x_i)\dt +o(\dt),\qquad i=1,2,
$$
is negative and bounded away from zero in domain $\{x_i\geq 1, x_{i^{\star}}\geq C_{i^{\star}}\}$, $i=1,2$, where both~$C_1$ and $C_2$ are large enough. Now Lemma~\ref{lem7.3.6} yields that in a finite mean time the Markov chain hits the boundary.
\end{proof}
 
\begin{remark}
\label{R0}
{\rm  Note that in the case of the competition process with Lotka-Volterra interaction (mentioned in the introduction),
the lemma follows from \cite[Theorem 5]{Reuter}. }
\end{remark}
Let us finish the proof of the theorem. Let $T_j$ be the duration of $j$-th visit to set $D_N=\{x_1>N, x_2\leq 1\}\cup\{x_1\leq 1, x_2>N\}$, where $N$ is chosen in the proof of Lemma~\ref{lem_stay@axesT1}. This lemma yields that $\P(T_j<\infty)\leq 1-\eps$ on $\{T_{j-1}<\infty\}$. Consequently, with probability one, $T_j<\infty$ only for finitely many $j$, and the process eventually confines to set $D_N$.

Finally, suppose for definiteness that the absorbing set is $\{x_1>N, x_2\leq 1\}$. Since the drift of $\xi_2(t)$ at $x_2=1$ is directed down, the process eventually jumps from level $x_2=1$ to level $x_2=0$. On the other hand, the process cannot stay forever at axis $x_2=0$ as $\l_2>0$. Thus, the Markov chain goes to infinity oscillating between levels $x_2=0$ and $x_2=1$ as claimed. Theorem~\ref{T1} is proved.

\subsection{Proof of Theorem~\ref{T2}}\label{ProofT2}
We start with the following lemma which is similar to Lemma~\ref{lem_stay@axesT1}. 

\begin{lemma}\label{lem_stay@axesT2}
There exists $\eps>0$, depending on the model parameters only, such that
$$
\inf_{x_1\geq 0}\P\left(\tilde A_1|\xi(0)=(x_1, 0)\right)\ge \eps\, \text { and }
\inf_{x_2\geq 0}\P\left(\tilde A_2|\xi(0)=(0, x_2)\right)\ge \eps,
$$
where 
$$
\tilde A_i=\begin{cases}
\xi_i(t)\to\infty \text{ and } \xi_{i^{\star}}(t)\in\{0,1\}, \forall t\ge 0,& \text{ if } \a_{i}>0,\\
\xi_i(t)\to\infty \text{ and } \xi_{i^{\star}}(t)\in\{0,1,2\}, \forall t\ge 0, & \text{ if } \a_{i}=0.
\end{cases}
$$
\end{lemma}

\begin{proof}
Denote $x=x_1$ and $y=x_2$ for simplicity of notations. We prove the lemma only in the case $\xi(0)=(x, 0)$. The proof in the case, where the initial position of the process is on the other axis, is identical.
 
First, assume that $\a_1>0$. Consider function $f$ defined in~\eqref{f1/3} with parameters $\mu$ and $\nu$ such that 
$$
0<\nu<\mu<1.
$$
Let ~$\Gen$ be the generator of the competition process with linear interaction. Given $x>0$ transitions from state $(x, 0)$ are possible only to states $(x+1, 0)$ and $(x, 1)$. These transitions occur with rates $\l_1+\a_1x$ and $\l_2$ respectively. 
Using equation~\eqref{Taylor} we obtain that 
\begin{equation}
\label{Gen11}
\begin{split}
\Gen f(x,0)&=(\l_1+\a_1x)\left((x+1)^{-\nu}-x^{-\nu}-(x+1)^{-\mu}+x^{-\mu}\right)
+\l_2 x^{-\mu}\\
&=-\nu\a_1x^{-\nu} +(\mu\a_1+\l_2)x^{-\mu}+o\left(x^{-\nu}\right)+o\left(x^{-\mu}\right)\leq 0,
\end{split}
\end{equation}
for sufficiently large $x>0$, as $\nu<\mu$.

Now, given that $x>0$, the transitions from state $(x, 1)$ to states $(x+1, 1)$, $(x-1, 1)$, $(x, 2)$ and $(x, 0)$ occur with rates $\l_1+\a_1x$, $\b_1$, $\l_2+\a_2$ and $\b_2x$ respectively. Therefore, using equation~\eqref{Taylor} one more time we obtain that 
\begin{equation}
\label{Gen12}
\begin{split}
\Gen f(x,1)&=(\l_1+\a_1x)\left((x+1)^{-\nu}-x^{-\nu}\right) +\b_1\left((x-1)^{-\nu}-x^{-\nu}\right)\\
&
+(\l_2+\a_2)\left(1-x^{-\nu}\right)-\b_2x^{1-\mu}\\
&\leq -\nu\a_1x^{-\nu}+\l_2+\a_2-\b_2 x^{1-\mu}+o(x^{-\nu})\leq 0,
\end{split}
\end{equation}
for all sufficiently large $x$, as $\mu<1$. 

Next, given $N>0$ define 
\begin{equation*}
\sigma=\inf(t: \xi(t)\notin\{x>N,\, y=0, 1\}).
\end{equation*} 
Assume that~$N$ is so large that the bounds~\eqref{Gen11}) and~\eqref{Gen12} hold for $x>N$. Then $Z(t)=f(\xi(t\wedge \sigma))$ is a non-negative supermartingale. The proof can be finished by using the argument based on the optional stopping theorem,  in a manner similar to the proof of Lemma~\ref{lem_stay@axesT1}.

Assume now that $\a_1=0$. In this case, instead of function~\eqref{f1/3} we consider the following function
$$
g(x,y)=\begin{cases}
\frac1{\ln x}-\frac1{\ln^3 x}
-\frac{\l_1/\l_2}{x\ln^2 x}
+\frac{1}{x\ln^3 x}
, &\text{if } y=0;\\
\frac1{\ln x}-\frac1{\ln^3 x}, &\text{if } y=1;\\
\frac1{\ln x}, &\text{if } y=2;\\
1, &\text{if } y\ge 3.
\end{cases}
$$
Using Taylor's expansion, we obtain that 
\begin{align}
\label{Gen20}
\begin{split}
\Gen g(x, 0)&\leq -\frac{\l_2}{x\ln^3 x}+ O\left(\frac1{x\ln^4 x} \right)\leq 0,\\
\Gen g(x, 1)&\leq -\frac{\b_2 \l_1/\l_2}{\ln^2 x}
+ O \left( \frac1{\ln^{3}x} \right)\leq 0,\\
\Gen g(x, 2)&\leq 
-\frac{\b_2 x}{\ln^3 x}+ O \left( 1 \right)\leq 0,
\end{split}
\end{align}
for all sufficiently large $x$. The rest of the proof is analogous to the proof in case $\a_1>0$ above, and we skip the details.
\end{proof} 

The other key ingredient of the proof is the following lemma, which is verbatim of Lemma~\ref{L2} in the proof of Theorem~\ref{T1}.
 
\begin{lemma}\label{lem_get2axesT2}
Define $\tau=\inf\{t:\xi_1(t)=0\text{ or }\xi_2(t)=0\}$. Then $\tau$ is a.s.\ finite.
\end{lemma}
Lemma~\ref{lem_get2axesT2} is proved in Section~\ref{Prooflem_get2axesT2}. Similarly to the proof of Theorem~\ref{T1}, it follows from Lemma~\ref{lem_stay@axesT2} and Lemma~\ref{lem_get2axesT2} that, with probability~$1$, the process eventually confines either to set $\{x_2\leq \kappa_2\}$, or to set $\{x_1\leq \kappa_1\}$, where $\kappa_i,\, i=1,2$ are defined in~\eqref{Li}. Suppose now for definiteness that the absorbing set is $\{x_2\leq \kappa_2\}$ and consider the following two cases. First, suppose that~$\a_1>0$, so that~$\kappa_2=1$. In this case the process cannot stay forever at line~$x_2=0$. Indeed, let $(a_{j}, 0),\, j\geq 1$ be a sequence of points successively visited by the Markov chain on line $x_2=0$. The probability of jump $(a_{j}, 0)\to (a_{j}, 1)$ can be bounded below by~$O(1)/(a_1+j)$ (for instance, consider the worst case scenario, when the process always jumps to the right); therefore, by the conditional Borel-Cantelli lemma, there are infinitely many jumps from line~$x_2=0$ to line~$x_2=1$. Combining this with Lemma~\ref{lem_get2axesT2}, or, simply noting that the probability of a jump from line~$x_2=1$ to line~$x_2=0$ is bounded below (it tends to $\b_2/(\a_1+\b_1+\b_2)$ as $x\to \infty$) one can conclude that the process cannot stay forever at line~$x_2=1$ as well; hence, it goes to infinity oscillating between lines~$x_2=0$ and~$x_2=1$, as claimed. 

Finally, suppose that $\a_1=0$ in which case~$\kappa_2=2$. The probability of transition $(x_1,0)\to (x_1,1)$ is equal to $\l_2/(\l_1+\l_2)$ for all~$x_1$, so that the Markov chain cannot forever stay at~$x_2=0$. Similarly to the above, let $(a_{j}, 1),\, j\geq 1$ be a sequence of points successively visited by the Markov chain on line $x_2=1$. The probability of jump $(a_{j}, 1)\to (a_{j}, 2)$ can be bounded below by~$O(1)/(a_1+j)$. Again, by the conditional Borel-Cantelli lemma, there are infinitely many jumps from line~$x_2=1$ to line~$x_2=2$. Combining this with Lemma~\ref{lem_get2axesT2}, or, simply noting that probabilities of jumps both from line $x_2=2$ to line $x_2=1$, and from line $x_2=1$ to line $x_2=0$, are bounded below by constants, we obtain that the process goes to infinity oscillating between lines~$x_2=0$ and~$x_2=2$, as described.

\subsection{Proof of Lemma~\ref{lem_get2axesT2}}
\label{Prooflem_get2axesT2}

Note that each of the following lines $x_2=\frac{\a_1x_1+\l_1}{\b_1}$ (line~$l_1$) and $x_2=\frac{\b_2x_1-\l_2}{\a_2}$ 
(line~$l_2$) divides~$\Z_+^2$ into two parts. The infinitesimal drift of $\xi_1(t)$ is negative above the line $l_1$, and positive below it. Similarly, the infinitesimal drift of~$\xi_2(t)$ is negative below line~$l_2$ and positive above it. There are two cases of mutual location of lines~$l_1$ and~$l_2$, namely, $\a_1\a_2 < \b_1 \b_2$ and $\a_1\a_2>\b_1 \b_2$.

If $\a_1\a_2 < \b_1 \b_2$, then line~$l_2$ is located above line~$l_1$ in the positive quarter plane. Both process components have negative drift in the domain between the lines. Moreover, the drift of one of the process components remains negative outside the negative cone. Consequently, with probability $1$, the process eventually hits the axes. The formal proof is similar to the case of competition processes with non-linear interaction in Theorem~\ref{T1}, therefore we skip the details. In addition, finiteness of the hitting time in this case follows from results in~\cite{Reuter} (see Section~\ref{appendix}). 

The case $\a_1\a_2\geq \b_1 \b_2$ is different from the previously considered cases. In order to explain this, assume for a moment that $\a_1\a_2>\b_1\b_2$. Then there is a positive drift in {\em both} coordinates in the domain between lines $l_1$ and $l_2$. If the process starts outside the domain, where the drift of the smallest component is strictly negative, then this component becomes zero in a finite mean time by the same reasoning as in all previous cases. However, if the initial position of the process is inside the domain, then one has to show that the process eventually leaves the domain.
 
The proof of the lemma in this case is given in Section~\ref{alpha>beta}. The proof is based on an appropriately modified method used in \cite{Freedman} for analysis of Friedman's urn model. The main idea of the original method is explained in Section~\ref{Freed}.

\subsubsection{Freedman's method for Friedman's urn model}
\label{Freed}

In this section we explain the main idea of Freedman's method for Friedman's urn model. First, recall that Friedman's urn model with parameters $\a\geq0$ and $\b\geq 0$ describes a DTMC $(W_n, B_n)\in \R^2_{+}\setminus (0, 0)$ evolving as follows. Given $(W_n, B_n)=(W,B)$ the Markov chain jumps to $(W+\a, B+\b)$ with probability $W/(W+B)$, and to $(W+\b, B+\a)$ with probability $B/(W+B)$. In order to demonstrate the main idea of the method we are going to consider another Markov chain (the auxiliary process) instead. The auxiliary process is a DTMC $(X_n, Y_n)\in \Z_{+}^2\setminus (0, 0)$ evolving as follows. Given $(X_n, Y_n)=(x,y)$ it jumps to states $(x+1, y)$ and $(x, y+1)$ with probabilities $\frac{\a x+\b y}{(\a+\b)(x+y)}$ and $\frac{\a y+\b x}{(\a+\b)(x+y)}$ respectively. Similar to competition processes with linear interaction, DTMC $(X_n, Y_n)$ takes values in the integer quarter plane and jumps to the nearest neighbour states. There is also a certain similarity between transition probabilities of $(X_n, Y_n)$ and the competition processes with linear interaction, although the interaction between $X_n$ and $Y_n$ can now be regarded as cooperative rather than competitive. Furthermore, the auxiliary process and Friedman's urn model are closely related, since
\begin{equation*}
\begin{split}
W_n&=\a X_n+\b Y_n\\
B_n&=\b X_n +\a Y_n.
\end{split}
\end{equation*}
In other words, $X_n$ ($Y_n$ resp.) can be viewed as the number of times a white (black resp.) colour has been picked up in Friedman's urn model by time $n$. Without loss of generality, we apply the Freedman's method to the auxiliary process $(X_n, Y_n)$. Given $\a\geq 0$ and $\b\geq 0$ define
\begin{equation}
\label{rho}
\rho=\frac{\a-\b}{\a+\b}.
\end{equation}
Theorem~\ref{TT} below describes the asymptotic behaviour of the auxiliary process under certain assumptions. The theorem is almost a verbatim copy of a part of Theorem 3.1 in \cite{Freedman} for the original Friedman's urn model with parameters $\a$ and $\b$. We state and prove the theorem for the auxiliary process for the following reason. There is certain similarity between our competition process and the auxiliary process, which allows to adapt the idea of the proof of Theorem~\ref{TT} for our purposes, therefore we provide the proof here for the readers convenience.

\begin{theorem}
\label{TT}
If $\rho>1/2$ then $n^{-\rho}(X_n-Y_n)$ converges almost surely to a non-trivial random variable.
\end{theorem}

\begin{proof}
Define the difference between the components $X_n$ and $Y_n$ as $U_n=X_n-Y_n$, and their total amount as $S_n=X_n+Y_n$; note that $S_n=S_0+n$. We have
\begin{equation}
\label{recur}
\begin{split}
\E\left(U_{n+1}|U_n\right)&
=U_n\left(1+\frac{\a-\b}{(\a+\b)S_n}\right)=U_n\left(1+\frac{\a-\b}{s+(\a+\b)n}\right)\\
\E\left(U_{n+1}^2|U_n\right)&
=U_n^2\left(1+\frac{2(\a-\b)}{(\a+\b)S_n}\right)+1=U_n^2\left(1+\frac{2(\a-\b)}{s+(\a+\b)n}\right)+1,
\end{split}
\end{equation}
where $s=(\a+\b)S_0$.
Denote 
$$
a_n(j)=\left(1+\frac{(\a-\b)j}{s+(\a+\b)n}\right),\, j=1,2.
$$ 
In these notations we get that 
\begin{equation}
\label{recur1}
\begin{split}
\E\left(U_{n+1}|U_n\right)&=U_na_n(1),\\
\E\left(U_{n+1}^2|U_n\right)&=U_n^2a_n(2)+1.
\end{split}
\end{equation}
The first equation in the preceding display means that 
\begin{equation}
\label{Zn}
Z_n:=U_n\prod_{k=0}^{n-1}a_k^{-1}(1),\, n\geq 1,
\end{equation}
is a martingale. The second equation gives
$$
\E\left(U_{n+1}^2\right)=\E\left(U_n^2\right)a_n(2)+1.
$$
Using this identity recursively, we arrive at the following equation
$$
\E\left(U_{n+1}^2\right)=\left(U_0^2+\sum\limits_{j=0}^n\prod\limits_{k=0}^ja_k^{-1}(2)\right)\prod_{k=0}^{n}a_k(2).
$$
Note that 
$$
\prod_{k=0}^m a_k(j)=(C_j+o(1)) m^{j\rho},\quad j=1,2,
$$
for some $C_1, C_2>0$, so that 
$$
\sum\limits_{j=0}^{\infty}\prod\limits_{k=0}^ja_k^{-1}(2)<\infty,
$$
as $\rho>1/2$. Consequently, $\sup_{n}n^{-2\rho}\E(U_n^2)<\infty$. Now Doob's convergence theorem implies that martingale $Z_n$ defined in~\eqref{Zn} converges almost surely to a finite limit as $n\to \infty$. Theorem~\ref{TT} is thus proved. 
\end{proof}

\subsubsection{Proof of Lemma~\protect\ref{lem_get2axesT2} in case  \texorpdfstring{$\a_1\a_2>\b_1\b_2$}{Lg}}
\label{alpha>beta}

\begin{proof}[ Proof in the symmetric case.] We start with the symmetric case $\l_1=\l_2=\l$, $\a_1=\a_2=\a$ and $\b_1=\b_2=\b$, where $\a>\b$, in order to provide an intuition for the way how the proof works. Denote by $\zeta(n)=(\zeta_1(n), \zeta_2(n))\in \Z_{+}^2$, $n\in \Z_{+}$, the DTMC corresponding to CTMC~$\xi(t)$. Let $\{{\cal F}_n\}_{n=1}^\infty$ be the standard natural filtration associated with the Markov chain $\zeta(n)$. Define 
\begin{align}
\label{tau}
S_n&=\zeta_1(n)+\zeta_2(n),\quad 
U_n=\zeta_1(n)-\zeta_2(n),\quad
\tau=\min\{m: \zeta_1(m)=0 \text{ or } \zeta_2(m)=0\}.
\end{align}

Assume that $\P(\tau=\infty)>0$ and get a contradiction. First, observe that
\begin{equation}
\label{step}
\E\left(U_{n+1}^2|{\cal F}_n\right)
=U_n^2\left(1+\frac{2(\a+\b)}{2\l +(\a+\b)S_n}\right)+1 \quad \text{on the event} \quad \{ \tau>n\}.
\end{equation}

\begin{remark}
\label{R1}
{\rm 
This expression is quite similar to the second equation in~\eqref{recur}; the fundamental difference is that the sum of the components, i.e.\ $S_n$, is now a {\it random} process. This is in contrast to both the auxiliary process and Friedman's urn model (as well as to other urn models without ball removals), where the sum of the components is a deterministic, usually linear, function of $n$. The main idea of what follows below is that the long-term behaviour of $S_n$ can be effectively controlled due to its simple asymptotic behaviour. 
}
\end{remark}

Trivially, $S_{n+1}-S_n=\pm 1$ and 
\begin{align*}
\P(S_{n+1}=S_n+1|S_n)&=\frac{\l+\a S_n}{2\l+(\a+\b)S_n},\\
\P(S_{n+1}=S_n-1|S_n)&=\frac{\b S_n}{2\l+(\a+\b)S_n}.
\end{align*}
The preceding display shows that the long term behaviour of $S_n$ is similar to a homogeneous simple random walk that jumps right and left with probabilities~$\frac{\a}{\a+\b}$ and~$\frac{\b}{\a+\b}$ respectively. Therefore, the strong law of large numbers, with some variations\footnote{A rigorous proof can be found further in Lemma~\ref{LemLLN}.}, implies that for any $\eps ,\delta > 0$ there exists $N$ such that
\begin{equation}
\label{srw}
\P\left(S_n\in [(\rho-\delta)n, (\rho +\delta)n], \,\, \forall\, n\geq N\right)\geq 1-\eps,
\end{equation}
where $\rho$ is defined in~\eqref{rho}.

Further, fix some $\delta>0$ such that $\rho+\delta<1$ and an arbitrary $\eps>0$; according to~\eqref{srw} there exists an $N=N(\eps)$ so large that
$$
\sigma_N=\min\left(n>N:\ S_n\notin [(\rho-\delta)n, (\rho +\delta)n] \right).
$$
satisfies 
\begin{align}\label{defsigman}
\P(\sigma_N=\infty)\ge 1-\eps.
\end{align} 
It follows from equation~\eqref{step} and the definition of $\sigma_N$ that 
\begin{align}
\label{eqUn}
\E\left(U_{n+1}^2|{\cal F}_n\right) &\geq U_{n}^2b_n 
\quad \text{on} \quad \left\{N\le n<\min\left(\sigma_N, \tau\right)\right\},
\\
\text{where }
b_n&=1+\frac{2(\a+\b)}{2\l +(\a+\b)(\rho+\delta)n}
=1+\frac{2 n^{-1}}{\rho+\delta} +O\left(n^{-2}\right).
\nonumber
\end{align}
Iterating~\eqref{eqUn} gives that
$$
\E\left(U_{n+1}^2|{\cal F}_N\right)\geq b_nb_{n-1}\dots b_{N+1} b_{N} U_{N}^2 \quad \text{on} \quad 
\left\{N\le n<\min\left(\sigma_N, \tau\right)\right\}.
$$
Assume $n\ge N$ everywhere below. Then
\begin{align*}
\prod_{k=N}^n b_k&=
\prod_{k=N}^n e^{\frac{2 k^{-1}}{\rho+\delta} +O\left(\frac 1{k^2}\right)}
=
e^{\sum_{k=N}^n\left[\frac{2 k^{-1}}{\rho+\delta} +O\left(\frac 1{k^2}\right)\right]}
\geq C_2\cdot n^{\frac2{\rho+\delta}}
\end{align*}
for some $C_2=C_2(N)>0$, so that 
$$
\E\left(U_{n+1}^2|{\cal F}_N\right)\geq C_2\cdot n^{\frac2{\rho+\delta}}1_{\{n<\min(\sigma, \tau)\}}.
$$
Dividing both sides by $n^2$ and taking the expectation gives 
$$
\E\left(\frac{U_{n+1}^2}{n^2}\right)\geq C_2\, n^{\frac2{\rho+\delta}-2}\, \P(n<\min(\sigma_N, \tau)).
$$
The left hand side of the preceding display is uniformly bounded in $n$, as $|U_n|\leq |U_0|+n$ for all $n\geq 0$. On the other hand, bound $\rho+\delta<1$ implies that $n^{\frac2{\rho+\delta}-2}\to \infty$ as $n\to \infty$. Therefore, if $\lim_{n\to \infty}\P(n<\min(\sigma_N, \tau))=\P(\sigma_N=\infty, \tau=\infty)>0$, as asserted, then we get a contradiction. Consequently, $\P(\sigma_N=\infty, \tau=\infty)=0$ and
$$
\P(\tau=\infty)\le \P(\sigma_N=\infty, \tau=\infty)+\P(\sigma_N<\infty)\le\eps
$$
by~\eqref{defsigman}. Since $\eps>0$ is arbitrary, $\P(\tau=\infty)=0$. 

In turn, this means that the process hits the axes in a finite time a.s., as claimed.
\end{proof}

We are now going to extend the above argument on the general case. Without loss of generality \ assume from now on that
\begin{align}\label{a1a2}
\a_1 \ge\a_2.
\end{align}

Let us find the asymptotic equilibrium direction for $\zeta(n)$, which will be shown to be unstable later in the proof. Indeed, if we assume that both 
$\l_{1}=0$ and $\l_2=0$ (they contribute very little to the birth rates when $x$ and $y$ are large) then the slope of the drift of the vector field corresponding to our system is given by
$$
\frac{\a_2 y - \b_2 x}{\a_1 x - \b_1 y}.
$$
It coincides with the slope of the vector $(x,y)$ if and only if $x=ry$ where $r$ solves
\begin{equation}
\label{quad}
\frac 1r=\frac{\a_2-r\b_2}{r\a_1-\b_1} \quad \Longleftrightarrow \quad
\b_2 r^2+ (\a_1-\a_2)r -\b_1=0,
\end{equation}
Since $x,y\ge 0$, we choose $r$ to be the positive root of~\eqref{quad} which can be written as
\begin{equation*}
r=\frac{-(\a_1-\a_2)+D}{2\b_2}, 
\end{equation*}
where
$$
D=\sqrt{(\a_1-\a_2)^2+4\b_1\b_2}=
\sqrt{(\a_1+\a_2)^2+4(\b_1\b_2-\a_1\a_2)}.
$$
Note that equation~\eqref{quad} can be rewritten as follows
\begin{equation}
\label{quad1}
r=\frac{\b_1+r\a_2}{\a_1+r\b_2}.
\end{equation}
Define the following variables
\begin{equation}
\label{Rn}
\begin{split}
R(x,y)&=(\a_1+\b_2)x+(\a_2+\b_1)y+\l_1+\l_2, \\
R_n&=R(\zeta(n)),
\end{split}
\end{equation}
and 
\begin{equation}
\label{Un}
\begin{split}
U(x,y)&=x-ry-d,\\
 U_n&=U(\zeta(n)),
\end{split}
\end{equation}
where 
\begin{equation}
\label{d}
d=-\frac {2(\l_1-\l_2r)+\a_1+\b_2r^2}{2(\a_1+\b_2r)}
\end{equation}
Assume that $x,y>0$. Then
\begin{align}
\E(U_{n+1}^2 | \zeta(n)=(x,y))&=
(U_n+1)^2\frac{\l_1+\a_1x}{R_n}+(U_n-1)^2\frac{\b_1y}{R_n}\nonumber\\
&+(U_n-r)^2\frac{\l_2+\a_2y}{R_n}+(U_n+r)^2\frac{\b_2x}{R_n}\nonumber
\\ 
&=U_n^2+2(\a_1+r\b_2)U_n\frac{x-\frac{\b_1+r\a_2}{\a_1+r\b_2}y-d}{R_n}
+\frac{2U_nQ_1+Q_2(x,y)}{R_n}\nonumber\\
&=U_n^2\,\left[1+\frac{2(\a_1+r\b_2)}{R_n}\right]
+\frac{2U_nQ_1+Q_2(x,y)}{R_n}.\label{UnExpansion}
\end{align}
where we used equation~\eqref{quad1} to rewrite the second term in the third line of the preceding display and used the notations
$$
Q_1:=d(\a_1+\b_2r)+\l_1-r\l_2 \,\, \text{ and }\,\, 
Q_2(x,y):=(r^2\b_2+\a_1)x+(\b_1+r^2\a_2)y.
$$
Consequently,
\begin{align*}
2U_nQ_1+Q_2(x,y)&=
2(\a_1+\b_2r)\left(d+\frac{2(\l_1-\l_2r)+\a_1+\b_2r^2}{2(\a_1+\b_2r)}\right)x\\
&+\left(\b_1+\a_2r^2-2rd(\a_1+\b_2r)-2r(\l_1-\l_2r)\right)y+Q_3,
\end{align*}
where $Q_3=\l_1+\l_2r^2-2d(\l_1-\l_2r)-2d^2(\a_1+\b_2r)$. Note that the coefficient in front of $x$ on the right hand side of the preceding equation is equal to $0$ by the definition of $d$ in~\eqref{d}. Further, using again the definition of $d$, we simplify the coefficient in front of $y$ and arrive at the following equation 
$$
2U_nQ_1+Q_2(x,y)=(\b_1+\a_1r+\a_2r^2+\b_2r^3)y+Q_3.
$$
Note that $\b_1+\a_1r+\a_2r^2+\b_2r^3>0$, therefore, 
\begin{equation}
\label{Q1Q2}
2U_nQ_1+Q_2(x,y)>0,
\end{equation}
for all $y\geq y_0$, where $y_0$ is a value depending on the model parameters.

Equations~\eqref{UnExpansion} and~\eqref{Q1Q2} imply that 
\begin{equation}
\label{eqUsnos}
\E(U_{n+1}^2 |\zeta(n)=(x,y))\geq U_n^2\,\left(1+\frac{2(\a_1+r\b_2)}{R_n}\right), \quad \text{if } x>0 \text{ and } 
 y\geq y_0.
\end{equation}
Our next goal is to obtain an upper bound for the total transition rate $R_n$. Define 
\begin{align}
S(x,y)&=(\a_1\a_2+\b_1\b_2+2\a_2\b_2)x
+(\a_1\a_2+\b_1\b_2+2\a_1\b_1)y,\nonumber\\
 S_n&=S(\zeta(n)),\label{Sn}\\
T(x,y)&=\b_2 x+(r\b_2 +\a_1-\a_2)y,\nonumber\\
T_n&=T(\zeta(n)) \nonumber
\end{align}
and
\begin{equation}
\label{tilde-rho}
\tilde\rho=\a_1\a_2-\b_1\b_2=(\a_2-r\b_2)(\a_1+r\b_2)>0.
\end{equation}
\begin{remark}
{\rm Note that $U_n$, defined by~\eqref{Un}, functions as a measure of departure from the equilibrium; $R_n$ is the common denominator, $S_n$ is the (almost) constant drift term (see~\eqref{Sndrift}), while~$T_n$ is some sort of a remainder, up to a multiplying coefficient, as it will be clear later in the proof.}
\end{remark}

Now we want to write $R(x,y)$ defined in~\eqref{Rn} as a linear combination of $S(x,y)$, $T(x,y)$, and an extra constant. In order to find the unknown coefficients, observe that both $S$ and~$T$ are linear in $x$ and $y$ with $S(0,0)=T(0,0)=0$. Therefore, $R(x,y)=\l_1+\l_2+k\, S(x,y)+ l\, T(x,y)$ where $k$ and $l$ can be found by solving the elementary system of linear equations 
$$
\begin{cases}
\frac{\partial R(x,y)}{\partial x}=
k\, \frac{\partial S(x,y)}{\partial x}+
l\, \frac{\partial T(x,y)}{\partial x}
\\
\frac{\partial R(x,y)}{\partial y}=
k\, \frac{\partial S(x,y)}{\partial y}+
l\, \frac{\partial T(x,y)}{\partial y},
\end{cases}
$$
yielding 
$$
k=\frac{\a_1+r\b_2 }{\tilde\rho}>0, \quad
l=-\frac{ (\a_1\a_2+2\a_2\b_2+\b_1\b_2)r+\a_1\a_2+2\a_1\b_1+\b_1\b_2 }{\tilde\rho}<0.
$$
Hence,
\begin{equation}
\label{RnSnTn}
R_n=(\l_1+\l_2)+\frac{\a_1+r\b_2 }{\tilde\rho}S_n + l T_n.
\end{equation}

The next statement is probably known, but just in case we present its proof here as well. 

\begin{lemma}\label{LemLLN}
Suppose that we are given a process $Z_n$ adapted to the filtration $\F_n$ such that $|Z_{n+1}-Z_n|\le B$ for all $n$ and
$$
a\le \E(Z_{n+1}-Z_n | \F_n)\le a + \frac{\sigma}{Z_n}
$$
for some constants $B>0$, $a>0$ and $\sigma\ge 0$. Then $Z_n/n\to a$ a.s.
\end{lemma}

\begin{proof}
Fix an $\eps>0$ and let $\hat Z_n=Z_n-a n$. Then $\hat Z_n$ is a submartingale with jumps bounded by $B+a$, and hence by Azuma-Hoeffding inequality 
\begin{align*}
\P(\hat Z_n-\hat Z_0\le -\eps n)\le \exp\left\{-\frac{\eps^2 n}{2(B+a)^2}\right\}
\end{align*}
and by Borel-Cantelli lemma the event $\{\hat Z_n/n\le -\eps+\hat S_0/n\}$ occurs finitely often. Since $\eps>0$ is arbitrary and $\hat Z_0/n\to 0$ we get that $\liminf_{n\to\infty} \hat Z_n/n\ge 0$ yielding $\liminf_{n\to\infty} Z_n/n\ge a$.

Next, define
$$
\bar Z_n=\hat Z_n -\sum_{i=1}^n \frac{\sigma}{\max\{1, Z_n\}}.
$$
On the event $\{Z_n\ge 1\}$ we have $\E(\bar Z_{n+1}-\bar Z_n | \F_n)=0$. Fix a large $N$ and consider $\bar Z_{n\wedge \tau_N}$ where $\tau_N=\inf\{n\ge N: Z_n<1\}$. Then $\bar Z_{n\wedge \tau_N}$ is a martingale for $n\ge N$ with jumps bounded by $B+a+1$, and applying Azuma-Hoeffding inequality again we get 
\begin{align*}
\P(\left|\bar Z_{n\wedge \tau_N}-\bar Z_N\right| \ge \eps n)\le 2\exp\left\{-\frac{\eps^2 (n-N)}{2(B+a+1)^2}\right\}
\end{align*}
for any $\eps>0$. By an argument similar to the first part of the proof, this implies that $\lim_{n\to\infty} \bar Z_{n\wedge \tau_N}/n=0$ a.s. However, the first part of the proof implies that $\tau_N=\infty$ for all but finitely many $N$'s a.s. Hence $\lim_{n\to\infty} \bar Z_{n}/n=0$ a.s. Now, the fact that $\liminf_{n\to\infty} Z_n/n\ge a$ gives us that $\sum_{i=1}^n \frac{\sigma}{\max\{1, Z_n\}}\le O(\log n)$ so that $\bar Z_n -\hat Z_n=o(n)$ thus implying the statement of the lemma.
\end{proof}

\begin{prop}
\label{liminfW}
Consider $S_n$ and $\tilde\rho$ defined in~\eqref{Sn} and~\eqref{tilde-rho} respectively. Then $\lim_{n\to\infty} \frac{S_n}n=\tilde\rho$ a.s.
\end{prop}

\begin{proof}[Proof of Proposition~\ref{liminfW}.]
Note that the jumps of $S_n$ are bounded (they can take at most four distinct values). The expected drift of $S_n$ is given by
\begin{align*}
\E(S_{n+1}- S_n |\zeta(n)=(x,y))&=\frac{(\a_1\a_2+\b_1\b_2)
((\a_1-\b_2)x+(\a_2-\b_1)y)}{R_n}\\
&+2\frac{\a_2\b_2(\a_1x-\b_1y)+\a_1\b_1(\a_2y-\b_2x)}{R_n}\\
&
+\frac{(\l_1+\l_2)(\a_1\a_2+\b_1\b_2) +2\l_1\a_2\b_2+2\l_2\a_1\b_1}{R_n}
\end{align*}
An easy algebraic computation gives that the sum of terms with $x$ in the first and the second numerators on the right hand side of the preceding display is equal to $\tilde\rho(\a_1+\b_2)x$. Similarly, the sum of all terms with $y$ in the same numerators is equal to $\tilde\rho(\a_2+\b_1)y$. Rearranging all terms with $\l_1$ and $\l_2$ in the last numerator of the same display gives the following quantity 
$$
\tilde\rho(\l_1+\l_2)+2\l_1\b_2(\a_2+\b_1)+2\l_2\b_1(\a_1+\b_2).
$$
Thus, we obtain that 
\begin{align}\label{Sndrift}
\E(S_{n+1}- S_n | \zeta(n)=(x,y))= \tilde\rho
+\frac{2\l_1\b_2(\a_2+\b_1)+2\l_2\b_1(\a_1+\b_2)}{R_n}\ge \tilde\rho>0.
\end{align}
Note that $R(x,y)\ge (x+y)\, \min\{\b_1,\b_2\} $ and 
$$
S(x,y)\le (x+y)\, \max\{\a_1\a_2+\b_1\b_2+2\a_2\b_2, \a_1\a_2+\b_1\b_2+2\a_1\b_1\}
$$
and since $\b_1,\b_2>0$ we have $R(x,y)\ge C_1 S(x,y)$ for some positive constant $C_1$, so that $R_n\ge C_1 S_n$. 
Now the result follows from Lemma~\ref{LemLLN} with $a=\tilde \rho$.
\end{proof}

\begin{corollary}
\label{corS}
Let $\kappa=\liminf_{n\to\infty} \frac{T_n}n$. Then $\P(\kappa>0)=1$.
\end{corollary}

\begin{proof}
Similarly to the preceding proof, 
$$
T(x,y)\ge (x+y)\min(\b_2,r\b_2+\a_1-\a_2)\ge (x+y)\b_2\min (1,r)
$$
since $\a_1-\a_2\ge 0$, and thus $T_n\ge C_2S_n$ for some $C_2>0$. Hence, by Proposition~\ref{liminfW},
$$
\liminf_{n\to\infty} \frac{T_n}{n}\ge C_2 \liminf_{n\to\infty} \frac{S_n}{n}=C_2 \tilde\rho>0.
$$
\end{proof}

\begin{prop}
\label{Rn_prop}
For every $\delta>0$ and $\eps>0$ there exists $N$ such that 
$$
\P\left(R_n\le \frac{\a_1+r\b_2}{1+\delta}n, \,\, \forall\, n\geq N\right)\geq 1-\eps.
$$
\end{prop}
\begin{proof}[Proof of Proposition~\ref{Rn_prop}.]
Using equation~\eqref{RnSnTn}, Proposition~\ref{liminfW} and Corollary~\ref{corS} we obtain that for sufficiently small $\delta>0$, sufficiently large $n$ and any fixed $\eps$ 
$$
R_n= (\l_1+\l_2)+ (\a_1+r\b_2)\frac{S_n}{\tilde\rho}-l T_n
\le (\l_1+\l_2)+ (\a_1+r\b_2)(1+\delta)n +l\, \frac{\kappa}{2}n,
$$
with probability at least $1-\eps$. Recall that $\kappa>0$ by Corollary~\ref{corS}, and that $l<0$. Let $\delta>0$ be so small that 
\begin{equation}
\label{delta}
(\l_1+\l_2)+ (\a_1+r\b_2)(1+\delta)n + l\, \frac{\kappa}{2}n\leq (\a_1+r\b_2)(1-\delta)n\leq (\a_1+r\b_2)
\frac{n}{1+\delta}.
\end{equation}
Thus, we obtain that, with probability at least $1-\eps$, 
$$
R_n\le \frac{\a_1+r\b_2}{1+\delta}n,
$$
for all sufficiently large $n$, as claimed.
\end{proof}

The rest of the proof is similar to the symmetric case, and we are going to explain briefly some minor modifications required. First, define 
\begin{equation}
\label{tau1}
\tau=\min\{n: \zeta_1(n)=0\text{ or }\zeta_2(n)<y_0\},
\end{equation}
where $y_0$ is such that thr bound~\eqref{eqUsnos} holds. Then, assume that $\P(\tau=\infty)=0$ and arrive at a contradiction. To this end, fix $\delta>0$ such that equation~\eqref{delta} holds, and, given $N>0$ define 
$$
\eta_N=\min\left\{n\ge N: R_n>\frac{\a_1+r\b_2}{1+\delta}n \right\}.
$$
Assume that $N$ is sufficiently large, so that probability $\P(\eta_N=\infty)$ is sufficiently close to~$1$ to ensure that $\P(\eta_N=\infty, \tau=\infty)>0$. Then Proposition~\ref{Rn_prop} implies that 
\begin{align*}
\E\left(U_{n+1}^2|{\cal F}_n\right) &\geq U_{n}^2a_n 
\quad \text{on} \quad \left\{n<\min\left(\eta_N, \tau\right)\right\},
\\
\text{where }
a_n&=1+\frac{2(1+\delta)}{n}.
\nonumber
\end{align*}
Similarly to the symmetric case, it can be shown by using the inequality in the preceding display that $\P(\eta_N=\infty, \tau=\infty)=0$. This contradicts the assumption that $\P(\tau=\infty)>0$.

Finally, it might happen that $\zeta_1(\tau)>0$ and $0<\zeta_2(\tau)=y_0-1$. In this case, observe that the probability of hitting the horizontal axis $\{(x,0), x\in \Z_{+}\}$ is bounded below uniformly over starting location $(x, y_0-1)$, $x\ge 1$. Indeed, 
\begin{align*}
& \P(\zeta(\tau+y_0-1)=(x,0) | \zeta(\tau)=(x,y_0-1))=\prod_{k=1}^{y_0-1} \frac{\b_2 x}{\l_1+\l_2+(\a_1+\b_2) x+(\a_2+\b_1) (y_0-k)}
\\ & \ge 
\prod_{k=0}^{y_0-1} \frac{\b_2 }{\l_1+\l_2+\a_1+\b_2 +(\a_2+\b_1) (y_0-k)}={\rm Const}(\l_1,\l_2,\a_1,\a_2,\b_1,\b_2,y_0)>0
\end{align*}
since $x\ge 1$. Consequently, with probability one, the process eventually hits the boundary.

\subsubsection{Proof of Lemma~\protect\ref{lem_get2axesT2} in case  \texorpdfstring{$\a_1\a_2=\b_1\b_2$}{Lg}}
\label{alpha=beta}
The proof will be very similar to the case $\a_1\a_2>\b_1\b_2$, so we provide only its sketch. Let $S(x,y)$, $R(x,y)$, $S_n$, $R_n$ and $\tilde \rho$ be the same as in the previous section. Note that $\tilde\rho=0$ in this case, so we need to find a replacement for Lemma~\ref{LemLLN}.

Observe that due to the fact that $\a_1\a_2=\b_1\b_2$ we have $\a_1>0$ and $\a_2>0$ since $\b_1\b_2>0$ and
$$
S(x,y)=2\a_2 (\a_1+\b_2) x + 2\a_1 (\a_2+\b_1) y,
\ \
R(x,y)= (\a_1+\b_2) x + (\a_2+\b_1) y+\l_1+\l_2
$$
so that $R(x,y)\ge \frac{ S(x,y) }{\max\{2\a_1,2\a_2\}}$. Then~\eqref{Sndrift} becomes
$$
\E(S_{n+1}-S_n | {\cal F}_n)=\frac{2\l_1\b_2(\a_2+\b_1)+2\l_2\b_1(\a_1+\b_2)}{R_n}\in\left[ 0, \frac{C_3}{S_n} \right] 
$$
for some $C_3\ge 0$. Therefore, $S_n$ can be majorized by a Lamperti random walk (see~\cite{MPW})
and hence by Theorem~3.2.7 in~\cite{MPW} we get that
$$
\limsup_{n\to\infty} \frac{\log S_n }{\log n}\le 1/2\qquad \text{a.s.}
$$
As a result, the statement of Proposition~\ref{Rn_prop} holds with the displayed formula replaced by 
$$
\P\left(R_n\le n^{1/2+\delta}, \,\, \forall\, n\geq N\right)\geq 1-\eps
$$
and by setting $\delta=1/6$, on the event $R_n\le n^{2/3}$ the RHS of~\eqref{eqUsnos} becomes
$$
 U_n^2\,\left(1+\frac{2(\a_1+r\b_2)}{n^{2/3}}\right)
$$
leading to contradiction similarly to the case $\a_1\a_2>\b_1\b_2$.

\section{Appendix}\label{appendix}
In this section we recall the definition of the competition process from~\cite{Reuter} and briefly analyse the applicability of some theorems from that paper to competition processes in ours. 
 
Recall that the competition process in~\cite{Reuter} is defined as a CTMC $X(t)=(X_1(t), x_2(t))\in \Z_{+}^2$ that evolves as follows. Given the state $(x_1, x_2)\in \Z_{+}^2$, the CTMC jumps to
\begin{align}
\label{reuter}
\begin{split}
(x_1+1,x_2) &\quad\text{with rate}\quad a(x_1, x_2),\\
(x_1,x_2+1) &\quad\text{with rate}\quad b(x_1, x_2),\\
(x_1-1,x_2) &\quad\text{with rate}\quad c(x_1, x_2)\quad\text{if}\quad x_1>0,\\
(x_1,x_2-1) &\quad\text{with rate}\quad d(x_1, x_2)\quad\text{if}\quad x_2>0,\\
(x_1-1,x_2+1) &\quad\text{with rate}\quad e(x_1, x_2)\quad\text{if}\quad x_1>0,\\
(x_1+1,x_2-1) &\quad\text{with rate}\quad f(x_1, x_2)\quad\text{if}\quad x_2>0,
\end{split}
\end{align}
where $a(x_1, x_2), \ldots, f(x_1, x_2)\geq 0$. Following \cite{Reuter}, let us assume that the Markov chain is regular in a sense that there exists 
exactly one associated transition matrix. For simplicity, we assume in addition that Markov chain $X(t)$ is irreducible, 
although in general there might be absorption states.

Define the following quantities 
\begin{align}\label{eqrs}
r_k&=\max_{\substack{x_1, x_2>0;\\ x_1+x_2=k}} [a(x_1, x_2)+b(x_1, x_2)],\\
s_k&=\min_{\substack{x_1, x_2>0;\\ x_1+x_2=k}} [c(x_1, x_2)+d(x_1, x_2)],\nonumber\\
\tau&=\inf(t\geq 0: X_1(t)=0 \text{ or } X_2(t)=0).\nonumber
\end{align}
It follows from Theorem 2 in~\cite{Reuter} that 
\begin{equation}
\label{A}
A:=\sum\limits_{k=2}^{\infty}\frac{s_2\ldots s_k}{r_2\ldots r_k}=\infty,
\end{equation}
is a sufficient condition for hitting time $\tau$ to be finite almost surely.

Consider, for simplicity, the competition process with linear interaction (with transition rates of type 2 defined in~\eqref{eqjumps2}) in the symmetric case, that is $\a_i=\a, \b_i=\b, \l_i=\l, i=1,2$. 
Then
$$
r_k=2\l+\a k \quad \text{and} \quad s_k=\b k,
$$
and it is easy to see that if $\a\le \b$ then 
$$
\frac{ s_2\ldots s_k}{r_2\ldots r_k}\geq 
\begin{cases}
 C_1\left(\frac{\b}{\a}\right)^k, &\text{if }\a<\b,\\
 \frac{C_2}{k^{{2\l/\a}}}, &\text{if }\a=\b,
\end{cases}
$$
for some $C_1,C_2>0$ for all sufficiently large $k$. Consequently, if $\a<\b$ or $\a=\b<2\l$ then $A=\infty$; hence $\tau$ is almost surely finite. However, if $\a>\b$ or $\a=\b\ge 2\l$, then the results of \cite{Reuter} are not applicable.

Further, we are going to compare the long term behaviour of two simple competition processes.
One process of interest is the competition process $X(t)$ given in Example 2 in~\cite{Reuter}. This process is specified by the following choice of transition rates in~\eqref{reuter}
\begin{equation}
\label{Example2}
\begin{array}{llll}
a(x_1, x_2)&\equiv a, & b(x_1, x_2)\equiv b,& c(x_1, x_2)=\gamma
 x_1, \\ 
 d(x_1, x_2)&=\delta x_2, & e(x_1, x_2)=\eps x_1 x_2,& f(x_1,x_2)\equiv 0,
\end{array}
\end{equation}
where $a, b, \gamma, \delta, \eps>0$. The other process is a special case of the competition process with linear interaction which transition rates are specified by parameters $\a_1=\a_2=0$, $ \b_1=\delta, \b_2=\gamma, \l_1=a, \l_2=b>0$. In the introduction we interpreted such competition process as the OK Corral model with ``resurrection".

Interactions in these processes are different. However, their behaviours inside the quarter 
plane are quite similar. Indeed, the mean drift of each of these processes inside the domain are directed towards the axes. Further, quantities $r_k=a+b$, $k\geq 1$, and $s_k=k \min\{\gamma, \delta\},$ $k\geq 1$, are the same for both processes. Now either~\cite[Theorem 2]{Reuter}, or the argument based on \cite[Lemma 7.3.6]{MPW} (similar to Lemma~\ref{L2}) imply that $\tau<\infty$ a.s.\ in both cases. 

At the same time, these processes evolve differently, because of the difference in the transition rates on the boundary. The process with rates given by~\eqref{Example2} has a strong mean drift towards the origin, while an OK Corral type process jumps away with constant rate, as its death rates on the boundary are zero. This seemingly small change results in quite substantial difference in the long term behaviour of the processes. Indeed, define $\tilde r_k$ and $\tilde s_k$ by the same formula as $r_k$ and $s_k$ in~\eqref{eqrs} by taking the maximum (minimum resp.) over the set $x_1,x_2\ge 0$, that is, now we include the boundary states $(k, 0)$ and $(0, k)$. Theorem~4 in~\cite{Reuter} states that
$$
\tilde A=\sum\limits_{k=1}^{\infty}\frac{\tilde r_1\ldots \tilde r_{k-1}}{\tilde s_1\ldots \tilde s_k}<\infty
$$
is a sufficient condition for the competition process with transition rates~\eqref{reuter} to be positive recurrent, implying that the process governed by~\eqref{Example2} is positive recurrent. Indeed, $\tilde r_k=r_k=a+b>0$ and $\tilde s_k=s_k=k \min\{\g,\delta\}$, $k\ge1$, so,
$$
\tilde A=\frac{1}{a+b}\sum\limits_{k=1}^{\infty}\frac 1{k!}\left(\frac{a+b}{\min\{\g,\delta\}}\right)^k<\infty.
$$
(Note also that positive recurrence of this process follows from the Foster criterion for positive recurrence with Lyapunov function $f(x_1, x_2)=x_1+x_2$, but we skip further details). At the same time Theorem 4 from~\cite{Reuter} is not applicable to the OK Corall model with ``resurrection'', as $\tilde s_k=0$, while {\em our} Theorem~\ref{T2} shows that this process is transient and escapes to infinity in the only possible way, i.e. along the boundary, as described.

\section*{Acknowledgement}
SV research is partially supported by the Swedish Research Council grant VR2014-5147. We thank Mikhail Menshikov and Svante Janson for helpful discussions.


\begin{thebibliography}{1}
\bibitem{Anderson} Anderson, W. (1991). Continuous time Markov chains: an application oriented approach. Springer Verlag. 

\bibitem{BGT} Bingham, N. H., Goldie, C. M., and Teugels, J.L. (1987). Regular Variation. Cambridge University Press. 

\bibitem{Freedman} Freedman, David A. (1965). Bernard Friedman's urn. {\it Ann.\ Math.\ Statist.}~{\bf 36}, pp.~956--970. 

\bibitem{Iglehart1} Iglehart, D.L. (1964). Reversible Competition Processes. {\it Z. Wahrseheinliehkeitstheorie}~{\bf 2}, pp.~314--331.

\bibitem{Iglehart2} Iglehart, D.L. (1964). Multivariate competition processes.  {\it Ann. \ Math.\ Statist.}~{\bf 35}, pp.~350-361.

\bibitem{SVJ} Janson, S., Shcherbakov, V. and Volkov, S. (2019). Long term behaviour of a reversible system of interacting random walks. {\it Journal of Statistical Physics}~{\bf 175}, N1,  pp.~71--96.

\bibitem{Kingman} Kingman, J.~F.~C., and Volkov, S.~E. (2003). Solution to the OK Corral model via decoupling of Friedman's urn. {\it J.\ Theoret.\ Probab.}~{\bf 16}, pp.~267--276.

\bibitem{MMW} MacPhee, M.I., Menshikov, M.V., and Wade, A.R. (2010). Angular asymptotics for multi-dimensional non-homogeneous random walks with asymptotically zero drift.  {\it Markov Process. Related Fields}~{\bf 16}, Issue 2, pp.~351--388.

\bibitem{MPW} Menshikov, M.V., Popov, S. and Wade, A.R. (2017). Non-homogeneous Random Walks: Lyapunov Function Methods for Near-Critical Stochastic Systems. Cambridge University Press.

\bibitem{MS} Menshikov, M. and Shcherbakov, V. (2018). Long term behaviour of two interacting birth-and-death processes.  {\it Markov Process. Related Fields}~{\bf 24}, Issue 1, pp.~85--106. 

\bibitem{Pemantle}
Pemantle, R.(2007). A survey of random processes with reinforcement. {\it Probability Surveys}~{\bf 4}, pp.~1--79.

\bibitem{Reuter} Reuter, G. E. H. (1961). Competition processes. In: Neyman J. (Ed.) Proceedings of The Fourth Berkeley Symposium on Mathematical Statistics and Probability, v.II: Contributions to Probability Theory. University of California Press, Berkeley.

\bibitem{SV} Shcherbakov, V. and Volkov, S. (2015). Long term behaviour of locally interacting birth-and-death processes.
{\it Journal of Statistical Physics}~{\bf 158}, N1, pp.~132--157.
\end{thebibliography}
\end{document}